\documentclass[10pt]{amsart}
\usepackage{geometry}                
\usepackage{graphicx}
\usepackage{amssymb}

\usepackage{epstopdf}
\theoremstyle{theorem}
\newtheorem{lemma}{Lemma}[section]
\newtheorem{theorem}[lemma]{Theorem}
\newtheorem{prop}[lemma]{Proposition}
\newtheorem{corollary}[lemma]{Corollary}

\theoremstyle{remark}

\theoremstyle{definition}

\newcommand\beqnn{\begin{eqnarray*}}
\newcommand\eeqnn{\end{eqnarray*}}
\newcommand\beqn{\begin{eqnarray}}
\newcommand\eeqn{\end{eqnarray}}

\newcommand\bt{{\bar t}}

\newcommand\R{{\mathbb R}}

\newcommand\Ar{\mbox{Argmin}}
\newcommand\vp{\varphi}

\newcommand\bX{{\bar X}}
\newcommand\bT{{\bar \tau}}
\newcommand\tS{U}
\newcommand\eps{\epsilon}
\newcommand\Min{\mbox{Min}}
\newcommand\bvp{\bar \vp}
\newcommand\Net{{\mathcal N}}

\title{Renewal Structure of the Brownian Taut String}

\author{Emmanuel Schertzer}

\begin{document}
\maketitle

\begin{abstract}
In  a recent paper \cite{LS15},  M. Lifshits and E. Setterqvist 
introduced the taut string of a Brownian motion $w$, defined as the function of minimal quadratic energy on $[0,T]$
staying in a tube of fixed width $h>0$ around $w$. The authors
showed a Law of Large Number (L.L.N.) for the quadratic energy spent by the string for a large time $T$. 

In this note, 
we exhibit a natural renewal structure for the Brownian taut string, 
which is directly related to the time decomposition of the Brownian motion in terms of its $h$-extrema (as first introduced by
Neveu and Pitman \cite{NP89}). Using this renewal structure, we derive an expression 
for the constant  in the L.L.N. given in \cite{LS15}. In addition, we provide a Central Limit Theorem (C.L.T.) for the fluctuations of the energy
spent by the Brownian taut string. 
\end{abstract}


\section{Introduction and Main Results}
Let $AC([a,b])$ denote the set of absolutely continuous functions defined on $[a,b]$,
and $||f||_{\infty,[a,b]}$ denote the supremum of the function $|f|$ over $[a,b]$.
Given $T,h>0$ and a continuous function $w$, the taut string associated with $w$ is the function
such that for every strictly convex function  $c$, it is the unique solution
of the following minimization problem  
\beqn\label{Min-problem}
\Min\left\{ \int_0^T c\left(\vp'(u)\right) \ du \ : \ \vp\in \ \mbox{AC}([0,T]),\ \vp(0)=w(0) \ , \vp(T)  = w(T) ,   || w-  \vp ||_{\infty,[0,T]} \leq \frac{h}{2}  \right\}, \nonumber\\
\eeqn
Interestingly, the solution of the latter minimization problem does not depend on the choice of $c$ --- see Proposition \ref{prop-der} for more details.

In a recent paper \cite{LS15}, M. Lifshits and E. Setterqvist studied the long time behavior 
of the taut string constructed around the sample path of a Brownian motion. Using an argument based on a concentration
inequality for Gaussian processes, they showed that if
$w$ is a Wiener process sample path, then there exists  
a (non-explicit) constant ${\mathcal C}$ such that
\begin{equation}\label{limit}
\lim_{T\uparrow\infty} \ \frac{1}{{T}}  \ \int_0^T \ |\eta_T'(u)|^2 \ du  =  {\mathcal C} \ \ \mbox{a.s.,}
\end{equation}
where $\eta_T$ denotes the taut string associated with the path $w$ on the interval $[0,T]$. As they put it, the constant $\mathcal C$ ``shows how
much quadratic energy an absolutely continuous function must spend if it is bounded to stay within a certain distance from the trajectory of $w$''. The aim of this note is to provide
a generalization of their result. We will show that an analogous result holds for
a large class of penalization function $c$, i.e. that for a very general class of functions $c$, there exists ${\mathcal C}_c$ such that
\begin{equation}\label{limit}
\lim_{T\uparrow\infty} \ \frac{1}{{T}}  \ \int_0^T \ c(\eta_T'(u)) \ du  =  {\mathcal C}_c \ \ \mbox{a.s..}
\end{equation}
In addition, we provide 
a semi-explicit expression 
for the constant ${\mathcal C}_c$, and an estimate for the 
fluctuations of the energy around this limiting value.



Our result is based on a decomposition of the Brownian motion
in terms of its $h$-extrema that
was first proposed by
Neveu and Pitman \cite{NP89} and that we now expose.  
\begin{figure}\label{taut}
\includegraphics[scale=0.3]{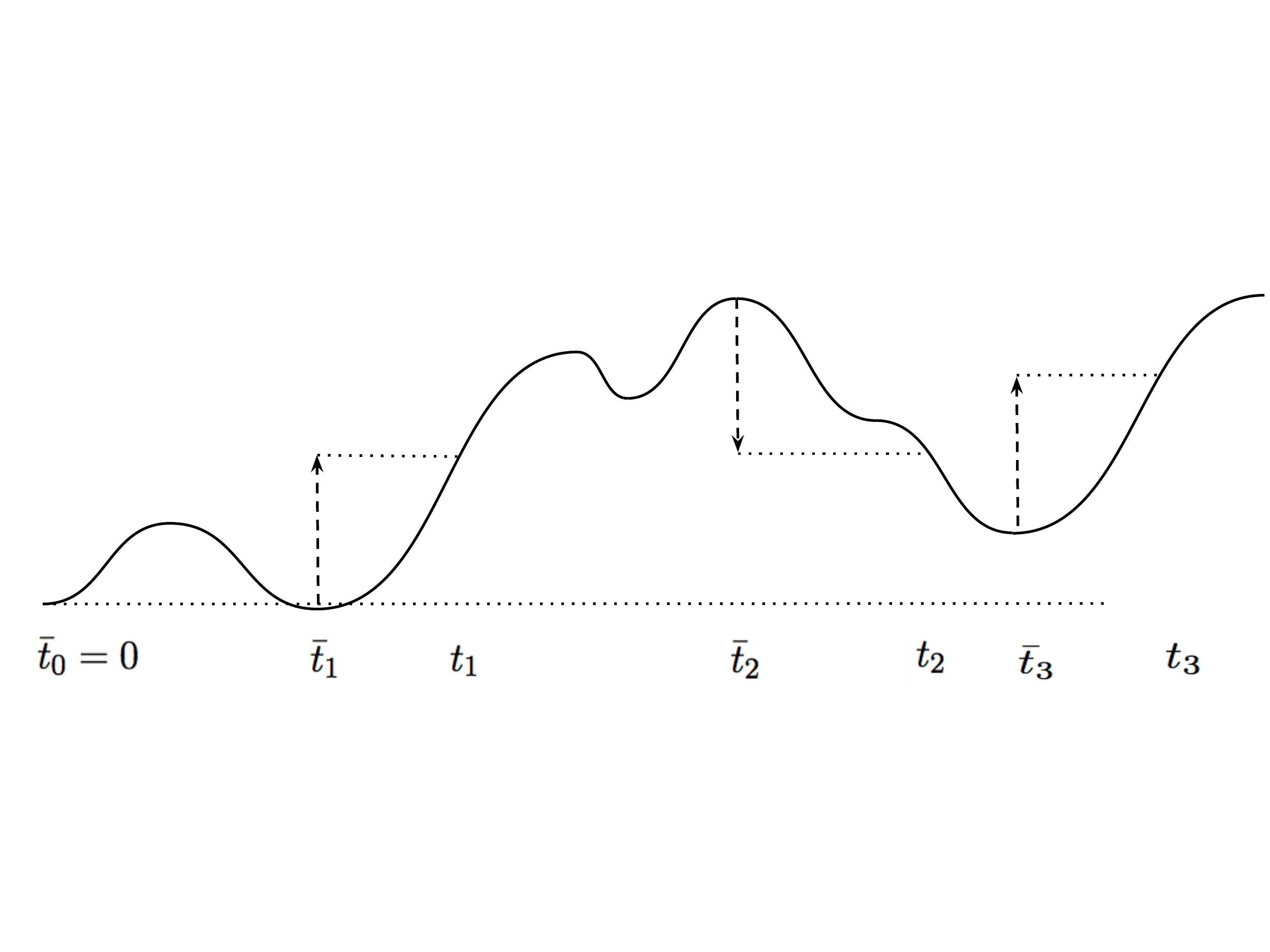}
\caption{Peaks and valleys of the function $w$. }
\end{figure}
Let us
introduce two sequences of times $\{t_n(w)\}_{n\geq0}\equiv \{t_n\}_{n\geq0}$ and $\{\bar t_n(w)\}_{n\geq0}\equiv \{\bar t_n\}_{n\geq0}$:
we first set $t_0,\bar t_0=0$, and for $n\geq0$,
\beqnn
t_{2n+1} & = & \inf\{t \geq t_{2n} :  w(t) - \inf_{[t_{2n},t]} w  = h \} \\
t_{2n+2} & = & \inf\{ t\geq t_{2n+1} : \sup_{[t_{2n+1},t]} w  - w(t) = h \}
\eeqnn
and
\beqnn
\bar t_{2n+1}    &  = & \sup\{ t\in[t_{2n},t_{2n+1}] \ : \  \inf_{[t_{2n},t]} w  = w(t)  \} \\
\bar t_{2n+2} & = & \sup\{t \in [t_{2n+1},t_{2n+2}] \ : \   \sup_{[t_{2n+1},t]} w = w(t) \} 
\eeqnn
In other words, $\bar t_{2n+1}$ (resp., $\bar t_{2n+2}$) is the starting time of the first upper (resp., lower) 
sub-excursion of height $\geq h$ after time $t_{2n}$ (resp., $t_{2n+1}$),  see Fig. \ref{taut}. 
Following the terminology of Neveu and Pitman \cite{NP89},
the
$w(\bar t_n)$'s correspond to the $h$-exterma 
of the function $w$, whereas the times $t_n$'s can be thought of as 
delimiting the successive peaks and valleys 
of height and depth greater than $h$ formed by the path $w$.  
In \cite{NP89}, the authors proved that this decomposition is natural 
for the Brownian motion: for a Wiener sample path $w$,
the sequence of paths $\{(-1)^n\left(w(\bt_n+u)-w(\bt_n)\right); \ 0 \leq u \leq \bt_{n+1}-\bt_{n})\}_{n\geq1}$ is a sequence
of i.i.d. random variables.



From Proposition \ref{ghk} in the Appendix, for every $i\geq1$, 
there exists a unique solution 
to the following minimization problem
\beqnn
\Min\{ \int_{\bt_i}^{\bt_{i+1}} \ |\vp'(u)|^2 \ du \ : \ \vp\in \ \mbox{AC}(I_i),\\
 \vp(\bt_i) = w(\bt_i) - (-1)^i \frac{h}{2} \ , \  \vp(\bt_{i+1}) = w(\bt_{i+1}) + (-1)^i \frac{h}{2},   \ || w-  \vp ||_{\infty, I_i} \leq \frac{h}{2} \},
\eeqnn
where $I_i:=[\bt_{i}, \bt_{i+1}]$.  We  denote this element by $\psi_i$.
The following theorem shows that the Neveu and Pitman's decomposition is also natural for the taut string.

\begin{theorem}\label{main}
Let $w$ be a continuous function (deterministic or random) and let $\eta_T$ be the associated taut string on $[0,T]$. 
Let
$$
N(T) \ = \ \sup\{n \ : \ \bt_n \leq T\}.
$$
If $N(T)\geq 4$,
the taut string $\eta_T$ and the function $\psi_i$ coincide on $[\bt_i, \bt_{i+1}]$ for every $2\leq i\leq N(T)-2$.  
\end{theorem}

We note in passing that the previous result implies the $\bt_i$'s  for $2\leq i\leq N(T)-1$
are knot points for the taut string, i.e., that the  string 
hits the boundary of the tube
at those times.

Let $w$ be the sample path of a Brownian motion. Since
the sequence
$$\{(-1)^n\left(w(\bt_n+u)-w(\bt_n)\right); \ 0 \leq u \leq \bt_{n+1}-\bt_{n})\}_{n\geq1}$$ 
is a sequence of i.i.d random variables, the latter 
result implies (at least informally) that the Brownian taut string on $[\bt_2, \bt_{N(T)}-1]$ is obtained by pasting together 
independent copies of the random path $\psi_1$ (up to a change of sign). 
Using standard limit theorems,
we can then leverage this intrinsic renewal structure
to derive
a L.L.N. and a C.L.T. for the long time behavior of the energy spent by Brownian taut string.

\begin{theorem}\label{main0}
Let $h,T>0$. Let $c$ be a locally bounded function, and $\alpha\geq0$ be such that 
$
\lim_{|x|\rightarrow\infty} c(x)/x^\alpha = 0.
$
Define
\beqnn
\tau & = & \bt_4 - \bt_2,  \\
{\mathcal E} & = & \int_{\bt_2}^{\bt_3} c\left( \psi_2'(u) \right) du +  \int_{\bt_3}^{\bt_4} c\left( \psi_3'(u) \right) du.
\eeqnn
If $w$ be the sample path of a Brownian motion, then
\begin{enumerate}
\item[(1)] $\tau$ and ${\mathcal E}$ have all finite moments.
\item[(2)]   $\frac{1}{T} \int_0^T \ c\left(\eta_T'(u)\right) \ du  \rightarrow E( {\mathcal E})/ E({\mathcal \tau})  \ \ \mbox{a.s..}$
\item[(3)]  
$$\frac{1}{\sqrt{T \ \hat \sigma^2 }}\left( \ \int_0^T \ c\left(\eta_T'(u)\right) \ du  \ - \ {T} \ \frac{E( {\mathcal E})}{E({\mathcal \tau})}    \right) \   \rightarrow  \ 
{\mathcal Z}\ \ \mbox{in distribution,}$$
\end{enumerate}
where ${\mathcal Z}$ is a standard normal random
variable and 
$$
\hat \sigma^2=\ \frac{(E( {\mathcal E} ))^2}{\left(E(\tau)\right)^3} \mbox{Var}(\tau) \ + \ \frac{\mbox{Var}({\mathcal E})}{E(\tau)} 
- 2 \mbox{Cov}(\tau, {\mathcal E})  \frac{E({\mathcal E})}{(E(\tau))^2}
$$
\end{theorem}


\bigskip

{\bf Outline of the paper.} The rest of the paper will be organized as follows. In Section \ref{proof:main}, 
we give a proof of Theorem \ref{main} and show some technical results that will be useful to
control the moments of ${\mathcal E}$ and $\tau$. In Section \ref{main0},
we give an outline of the proof of Theorem \ref{main0}, and postpone 
technical details to later sections. More precisely, in Section \ref{s:moments}, we control the moments of 
 ${\mathcal E}$ and $\tau$ and show some tightness results. In Section \ref{proof:CLTA}, we prove an extension of the C.L.T. for renewal processes.
 Finally, in the Appendix, we show some results related to the taut string in a deterministic setting.

\section{Proof of Theorem \ref{main}}
\label{proof:main}


In this section, we prove Theorem  \ref{main}. Along the way, we also prove a result
that  will be instrumental in controlling the moments of $\tau$ and ${\mathcal E}$.
In the following, $\bar c$ will denote an even, strictly convex function.

\begin{lemma}\label{lemma:2.1}
Let $f$ be a continuous function
such that $\bt_1(f)=0$. Let $V\geq t_1(f)$
and let $\Psi$
be the unique minimizer of
$$
\Min\{ \int_0^V \ \bar c(\vp'(u)) \ du \ : \ \vp\in \ \mbox{AC}([0,V]),   || f-  \vp ||_{\infty,[0,V]} \leq \frac{h}{2}  \}.
$$
Then $\Psi(0)=f(0)+\frac{h}{2}$
\end{lemma}

\begin{proof}
The existence and uniqeness of $\Psi$ are given by Proposition \ref{ghk} in the Appendix. Next, we need to show that for any absolutely continuous function $\vp$,
such that $||\vp-f||_{\infty,[0,V]}\leq h/2$, 
we can construct an absolutely continuous function $\bar \vp$ such that (1)  $\bvp$ stays in the tube of width $h$ 
around $f$ (i.e., such that $||\bar \vp-f||_{\infty,[0,V]}\leq h/2$), (2) we have the following boundary condition
\beqnn
\bar \vp(0)  =  f(0) + h/2
\eeqnn
and finally, (3) $\bar \vp$ has lower $\bar c$-energy:  $\int_{0}^{V} \bar c( \bar \vp'(u) ) du \ \leq \ \int_{0}^{V} \bar c( \vp'(u) ) du$.

We first claim that for every admissible function $\vp$ (i.e., $\vp$ is AC and satisfies condition (1) above), there exists $s\in[0, t_1(f)]$ such that 
$
\vp(s)=f(0)+h/2. 
$
This simply follows from the fact that 
for every $\vp\in AC([0,V])$
$$
\vp(0)\leq f(0) + h/2 \ \ \mbox{and} \ \ \vp(t_1(f)) \geq f(t_1(f)) - h/2 \ = \ f(0) + h/2, 
$$
where we used the fact that $\bar t_1(f)=0$.
This yields the existence of $s$ in $[0, t_{1}(f)]$ such that $
\vp(s)=f(0)+h/2 
$, as claimed earlier. 

Let us now consider the absolutely continuous $\bar \vp$ defined on $[0, V]$
as follows:
$$
\bar \vp(u) \ = \ 1_{u\in[0,s]} (f(0) + h /2) \ + \   1_{u\in[s, V]} \vp(u).
$$
From the definition of $t_1(f)$, it is straightforward to check that this function is guaranteed to stay in the tube 
of width $h$. Furthermore, since the minimum of $\bar c$ is attained at $0$, we have:
$$
\int_{0}^{V} \bar c( \bar \vp'(u) ) du \ \leq \ \int_{0}^{V} \bar c( \vp'(u) ) du,
$$ 
which ends the proof of our lemma.

\end{proof}

\begin{corollary}\label{lem:dec}
For every $i\geq1$, $\psi_i$ is the unique minimizer of the free-boundary problem
\beqnn
\Min\{ \int_{\bt_i}^{\bt_{i+1}} \ \bar c(\vp'(u)) \ du \ : \ \vp\in \ \mbox{AC}(I_i),  \ || w-  \vp ||_{\infty, I_i} \leq \frac{h}{2} \}, 
\eeqnn
\end{corollary}

\begin{proof}
Let $Q$ be the minimizer of the minimization problem described above (whose existence and uniqueness is again guaranteed by 
Proposition \ref{ghk}.). 
We claim that it is enough to show that
\beqn\label{esd}
Q(\bt_i) = w(\bt_i) - (-1)^{i} \frac{h}{2}, \ \mbox{and } Q(\bt_{i+1}) = w(\bt_{i+1}) + (-1)^{i} \frac{h}{2}.
\eeqn
In order to see this, we first note that those identities will directly imply that $Q$ coincides with the unique element in 
\beqnn
\Ar\{ \int_{\bt_i}^{\bt_{i+1}} \ \bar c(\vp'(u)) \ du \ : \ \vp\in \ \mbox{AC}(I_i),\\
 \vp(\bt_i) = w(\bt_i) - (-1)^i \frac{h}{2} \ , \  \vp(\bt_{i+1}) = w(\bt_{i+1}) + (-1)^i \frac{h}{2},   \ || w-  \vp ||_{\infty, I_i} \leq \frac{h}{2} \}.
\eeqnn
By Proposition \ref{prop-der} in the Appendix, this implies that $Q$ is also the unique minimizer of the minimization problem obtained
by replacing $\bar c$ by any strictly convex function $c$ in the latter expression. In particular, it is solution of the following minimization problem:
\beqnn
\Min\{ \int_{\bt_i}^{\bt_{i+1}} \ |\vp'(u)|^2 \ du \ : \ \vp\in \ \mbox{AC}(I_i),
\\ \vp(\bt_i) = w(\bt_i) - (-1)^i \frac{h}{2} \ , \  \vp(\bt_{i+1}) = w(\bt_{i+1}) + (-1)^i \frac{h}{2},   \ || w-  \vp ||_{\infty, I_i} \leq \frac{h}{2} \},
\eeqnn
which yields $\psi_i = Q$. 

It remains to show the two identities in (\ref{esd}).
We start by showing the first equality. Define $f(u)=(-1)^{i+1} \ w(u+\bt_i)$ for $u\in[0, \bt_{i+1}-\bt_i]$.
In the following, we will use the following notation
$$
\theta_i = t_i(f), \ \ \bar\theta_i = \bt_i(f)
$$
(whereas we recall that the $\bt_i$'s are defined with respect to the function $w$, i.e., $\bt_i\equiv\bt_i(w)$).
Note that under those notations, we have $\bar \theta_1=0$ and $\bar\theta_2 = \bt_{i+1}-\bt_i$.
Let us now consider $\Psi$ the unique minimizer of
$$
\Min\{ \int_0^{\bar\theta_2} \ \bar c(\vp'(u)) \ du \ : \ \vp\in \ \mbox{AC}([0,\bar \theta_2]),   || f-  \vp ||_{\infty,[0, \bar \theta_2]} \leq \frac{h}{2}  \}.
$$
and let $\tilde \Psi(u)=(-1)^{i+1} \Psi(u-\bt_i)$ for every $u\in[\bt_i, \bt_{i+1}]$. Using the fact that $\bar c$ is even, one can readily check from the definition that 
$\tilde \Psi$ belongs to
$$
\mbox{Argmin}\{ \int_{\bt_i}^{\bt_{i+1}} \ \bar c(\vp'(u)) \ du \ : \ \vp\in \ \mbox{AC}([\bt_i,\bt_{i+1}|]),   || w-  \vp ||_{\infty,[\bt_i,\bt_{i+1}]} \leq \frac{h}{2} \}
$$
and is therefore the unique minimizer of the corresponding variational problem.
As a consequence, it
coincides with
$Q$. On the other hand, since $\bar\theta_2>\theta_1$ and $\bar \theta_1=0$, the previous lemma implies
$
\Psi(0)=f(0)+\frac{h}{2}
$
and thus
$$
Q(\bt_i) \ = \ (-1)^{i+1} \Psi(0) \ = \ w(\bt_i) + (-1)^{i+1}\frac{h}{2}.
$$

Let us now turn to the second equality. Again, the idea is to use the previous lemma on a suitably chosen function.
Define $g(u)=(-1)^i w(\bt_{i+1}-u)$ on $[0,\bt_{i+1}-\bt_{i}]$. Finally, define
$\tau_i = t_i(g)$ and $\bar \tau_{i} \ = \ \bt_{i}(g)$. By construction, $\bar \tau_1=0$
and it is straightforward to check that $\bt_{i+1}-\bt_{i}\geq \tau_1$. Thus, denoting by $\hat \Psi$ the unique solution
of the minimization problem
$$
\Min\{ \int_0^{\bar t_{i+1}-\bar t_{i}} \ \bar c(\vp'(u)) \ du \ : \ \vp\in \ \mbox{AC}([0,\bar t_{i+1}-\bar t_{i}]),   || g-  \vp ||_{\infty,[0,\bar t_{i+1}-\bar t_{i}]} \leq \frac{h}{2}  \},
$$
we must have $\hat \Psi(0)=g(0)+\frac{h}{2}$. On the other hand,
the function $(-1)^{i}\hat \Psi(\bt_{i+1}-u)$ on the interval $[\bt_{i},\bt_{i+1}]$
coincides with $Q$ (by the same argument as above). Combining this fact with $\hat \Psi(0)=g(0)+\frac{h}{2}$
yields the desired result.

\end{proof}

\begin{proof}[Proof of Theorem \ref{main}.]
 Let us consider $\Phi$ defined on 
$[\bt_1,\infty)$ that coincides with $\psi_i$
on the interval $[\bt_i,\bt_{i+1})$ for $i\geq1$.
We claim that the function $\Phi$ is a solution of the following free-boundary minimization problem:
$$
 \Min\{ \int_{\bt_1}^{\bt_{N(T)}} \ |\vp'(u)|^2 \ du \ : \ \vp\in \ \mbox{AC}([0,T]),\ || w-  \vp ||_{\infty,[0,T]} \leq \frac{h}{2}  \}.
$$
for every $T$ such that  ${N(T)}\geq1$.
First, the condition
$|| w-  \Phi ||_{\infty,[0,T]}\leq h/2$ is obviously satisfied.
Further, $\Phi$ is absolutely continuous on $[\bt_1,\infty]$ because of 
the boundary conditions imposed on the $\psi$'s.
Furthermore,
for every $\vp\in AC([\bt_1,\bt_{N(T)}])$:
\beqnn
\int_{\bt_1}^{\bt_{N(T)}} |\vp'(u)|^2 du & = &  \sum_{i=1}^{{N(T)-1}} \int_{\bt_i}^{\bt_{i+1}} |\vp'(u)|^2 du \  \\
				 & \geq &   \sum_{i=1}^{{N(T)-1}} \int_{\bt_i}^{\bt_{i+1}} |\psi_i'(u)|^2 du \\
				  & = & \int_{\bt_1}^{\bt_{N(T)}} |\Phi'(u)|^2 du
\eeqnn
where the inequality follows from Corollary \ref{lem:dec} with $\bar c(x)=x^2$. Hence, $\Phi$ is a solution of the free-boundary minimization problem described above.

Next, we claim  that for every interval $[\bt_i, \bt_{i+1}], 1 \leq i \leq N(T)-1$, there must exist 
at least one $v_i\in[\bt_i, \bt_{i+1}]$ 
such that 
$\eta_T (v_i) \ = \ \Phi (v_i)$.
Let us first assume that $i$ is even. Since at $\bt_i$ (resp., $\bt_{i+1}$), the function $\psi_i$ is at the lower (resp., upper) boundary of the tube
of width $h$ surrounding $w$, we must have
$$
\eta_T(\bt_i) - \Phi(\bt_{i}) \geq 0, \ \ \mbox{and } \ \eta_T(\bt_{i+1}) -  \Phi(\bt_{i}) \leq 0.
$$
and the previous claim flows in the case where $i$ is even.
The odd case can be treated along the same lines.

Next, recall that the taut string $\eta_T$
is the unique solution of the minimization problem:
$$
\Min\{ \int_0^T \ |\vp'(u)|^2 \ du \ : \ \vp\in \ \mbox{AC}([0,T]),\ \vp(0)=w(0) \ , \vp(T)  = w(T) ,   || w-  \vp ||_{\infty,[0,T]} \leq \frac{h}{2}  \}. \nonumber\\
$$
Thus, the restriction of $\eta_T$ and  $\Phi$  on the interval $[v_1, v_{N(T)-1}]$
must be solution of the following minimization problem :
\beqnn
\Min\{ \int_{v_1}^{v_{N(T)-1}} \ |\vp'(u)|^2 \ du \ : \ \vp\in \ \mbox{AC}([v_1, v_{N(T)-1}]),\  \\ 
\vp(v_1)=\eta_{T}(v_1)=\Phi(v_1) \ ,  \ \  \vp(v_{N(T)-1})= \eta_T(v_{N(T)-1})=\Phi(v_{N(T)-1})  ,
\\   || w-  \vp ||_{\infty,[v_1,v_{N(T)-1}]} \leq \frac{h}{2}  \}.
\eeqnn
Since there is a unique solution to the latter minimization problem (again by Proposition \ref{ghk}), it follows that $\eta_T$ and $\Phi$
coincide on the interval $[v_1, v_{N(T)-1}]$. Finally, since $[\bt_2, \bt_{N(T)-1}] \ \subset \ [v_1, v_{N(T)-1}]$, 
this ends the proof of Theorem \ref{main0}.

\end{proof}


\section{Proof of Theorem \ref{main0}}

In this section, we give an outline of the proof of Theorem \ref{main0}, and postpone 
technical details to later sections. 
First, it is sufficient 
to show a weak version of our L.L.N.. Indeed, one can extend cour onvergence in probability statement to convergence a.s.  statement 
by using the same argument presented in Theorem 1.2. in \cite{LS15}.

When $N(T)\geq4$, Theorem \ref{main} implies that 
\begin{eqnarray*}
\int_0^T c(\eta_T'(u)) du \ = \ \int_{0}^{\bt_2} c(\eta_T'(u)) du   \ + \ \sum_{i=2}^{N(T)-2} \int_{\bt_i}^{\bt_{i+1}} c(\psi_i'(u)) du \ + \ \int_{\bt_{N(T)-1}}^T c(\eta_T'(u)) du.
\end{eqnarray*}
Let us define
$$
Y_i \ = \ \int_{\bt_{2i}}^{\bt_{2i+1}} c(\psi_{2i}'(u)) du +  \int_{\bt_{2i+1}}^{\bt_{2i+2}} c(\psi_{2i+1}'(u)) du
$$
and let ${\mathcal N}(T) = \sup\{n \ : \ \bt_{2n} \leq T\}$. 
Again assuming that $N(T)\geq4$, a little bit of algebra yields
$$
\int_0^T c(\eta_T'(u)) du \ = \  \sum_{i=1}^{\Net(T)-1} Y_i + R(T), \
$$
with 
\begin{eqnarray} R(T) & = & \int_{0}^{\bt_2} c(\eta_T'(u)) du \ + \ \int_{\bt_{N(t)-1}}^T  c(\eta_T'(u)) du \nonumber \\
&-& 1_{\{N(T) \mbox{ is even} \} } \int_{\bt_{N(T)-1}}^{\bt_{N(T)}}  c(\psi_{N(T)-1}'(u)) du. \label{R(T)}
\end{eqnarray}

In Section \ref{s:moments}, we show that $\{R(T)\}_{T\geq0}$ is tight  -- see Corollary \ref{cor:tightness} -- which implies that
$R(T)/\sqrt{T}$ and $R(T)/T$
converges to $0$ in probability as $T\rightarrow \infty$.
As a consequence, 
we only need to show the L.L.N. and the C.L.T. for 
the quantity
$
\sum_{i=1}^{\Net(T)-1} Y_i.
$
We start with the L.L.N. (second item of Theorem \ref{main0}). Write
$$
\frac{1}{T}\sum_{i=1}^{\Net(T)-1} Y_i. \ \ = \  \left(\frac{1}{\Net(T)}\sum_{i=1}^{\Net(T)-1} Y_i \right) \ \frac{\Net(T)}{T}.
$$
First,
$\{\bt_{2i+2}-\bt_{2i}\}_{i\geq1}$ is a sequence of i.i.d. random variables, and 
in Corollary \ref{cor:1}, we shall prove that its elements have all finite moments.
A standard renewal theorem implies that $ \frac{\Net(T)}{T}$ converges a.s. to ${1}/{E(\bt_4-\bt_2)}$.

Secondly, the sequence $\{(-1)^i(w(u+\bt_i)-w(\bt_i), \ u\in[0,\bt_{i+1}-\bt_i]\}_{i\geq1}$ is a sequence of i.i.d. random variables. This implies 
that $\{Y_i\}_{i\geq1}$ is also made of i.i.d. random variables.
Finally, we shall prove later that 
$Y_i$ has all finite moments -- see again Corollary \ref{cor:1} below. 
Our L.L.N. then follows by
a direct application of the strong L.L.N. and by noting that ${\mathcal N}(T)$ goes to $\infty$ as $T\rightarrow\infty$.

\bigskip

In order to prove our C.L.T., we will need to prove a result on renewal processes that we now expose. Let $\{(X_i,\tau_i)\}$ be an i.i.d. sequence 
of (possibly correlated) pairs of non-negative random variables with respective
finite non-zero expected value $\bar \tau$ and $\bar X$, 
finite and non-zero standard deviation $\sigma_\tau$ and $\sigma_X$ and covariance $\sigma_{X,\tau}$.
Define
$$
S_n \ = \ \sum_{k=1}^n \tau_k, \ \mbox{and} \ \ U_n \ = \  \sum_{k=1}^n X_k.
$$
and
$
\tilde{\mathcal N}(t) \ = \ \sup\{n\geq0 \ : \ S_n \leq t\}.
$
The following result is an extension of Anscombe Theorem \cite{A52}.

\begin{prop}\label{CTLA}
$$
\lim_{t\rightarrow\infty} \ \frac{\left( \ U_{\tilde{\mathcal N}(t)-1}  - t \bar X/ \bar \tau\ \right)}{\sqrt{t \bar \sigma^2}} \ = \ {\mathcal Z} \ \ \mbox{in distribution},
$$
where ${\mathcal Z}$ is a standard normal random variable and 
$$
\bar \sigma^2 \ = \ \frac{\bX^2}{\bT^3} \sigma_\tau^2 \ + \ \frac{\sigma_X^2}{\bT} - 2 \sigma_{X,\tau} \frac{\bX}{\bT^2}
$$
\end{prop}
Taking $X_i=\int_{\bt_{2i}}^{\bt_{2i+1}} c(\psi^{'}_{2i}(u)) du + \int_{\bt_{2i+1}}^{\bt_{2i+2}} c(\psi^{'}_{2i+1}(u)) du $ and $\tau_i = \bt_{2i+2}-\bt_{2i}$ in the latter proposition, yields the third part of Theorem \ref{main0}.

\section{Moment Estimates and Tightness.}\label{s:moments}

\subsection{Moments of ${\mathcal E}$ and $\tau$.} We start with some preliminary work. Define
$\sigma_0=0$ and for $n\geq1$,  
$$
\sigma_n \ = \ \inf\{t\geq \sigma_{n-1}\ : \ |w(t)-w(\sigma_{n-1})|=h/4 \}.
$$
By the strong Markov property, we note that $(\sigma_n-\sigma_{n-1}; \ n\geq1)$ is a sequence 
of i.i.d. random variables.

\begin{lemma}\label{moments}
For every $p\in\mathbb{Z}$, $E((\sigma_1)^p)<\infty$.
\end{lemma}
\begin{proof}
The case $p\in{\mathbb N}$ is well known. Let us focus instead on the case where $p<0$, i.e., for $n\in{\mathbb N}$,
let us prove that $E((1/\sigma_1)^n)<\infty$. First, 
\beqnn
E(1/(\sigma_1)^n) & = & n \int_{0}^\infty x^{n-1} P(1/\sigma_1>x) dx \\
		   & = &  n \int_{0}^\infty x^{n-1} P(\sigma_1<\frac{1}{x}) dx. 
\eeqnn
We then need to estimate the asymptotic behavior of $P(\sigma_1<y)$
when $y$ goes to $0$. Momentarily, we make the  dependence in $h$ explicit in the notations by writing $\sigma_n=\sigma_n^{h}$.
By Brownian scaling and symmetry, 
\beqnn
P(\sigma_1^h < y) & = & P(\sigma_{1}^{h/\sqrt{y}} <1) \\
			     & \leq & 2 P( \max_{[0,1]} w \  \geq \  \frac{h}{4\sqrt{y}}). 
\eeqnn
By standard large deviation estimates, 
$$
\lim_{y\rightarrow0} {16y/h^2} \ \log P( \max_{[0,1]} w \geq \frac{h}{4\sqrt{y}}) = - \frac{1}{2}.
$$
It follows that $P(\sigma_1^h<y)$ decreases exponentially fast to $0$ as $y\rightarrow0$, and thus that $E(1/(\sigma_1^h)^n)$ is finite.
\end{proof}

\bigskip

For $n\geq1$, define
$$
\Delta_n \ = \ \frac{4}{h} (w(\sigma_n)-w(\sigma_{n-1})).
$$
$\{\Delta_n, n\geq1\}$ is a sequence of i.i.d. random variables, independent of the $\sigma_n$'s, and 
$
P(\Delta_{n}=\pm1)=1/2.
$
From the sequence $\{\Delta_n\}_{n\geq1}$, define a sequence of integer
$\{n_k\}_{k\geq0}$ as follows: $n_0=0$ and 
$$
n_k \ = \ \inf\{i \geq n_{k-1}: \ \Delta_{4i}  = \Delta_{4i-1} =  \Delta_{4i-2} = \Delta_{4i-3}= (-1)^{k+1} \}.
$$
Note that  if $\Delta_{4n_k}=1$ (resp., $\Delta_{4n_k}=-1$), the path $w$ experiences 
an upcrossing (resp., downcrossing) of size $h$  on the interval $[\sigma_{4(n_{k}-1)},\sigma_{4n_k}]$.
In particular, on the interval $[0,\sigma_{4n_2}]$, the path $w$
must have experienced an upcrossing and then a downcrossing of size larger or
equal to $h$. 
See Fig. 2. This motivates the following lemma.

\begin{figure}\label{taut3}
\includegraphics[scale=0.3]{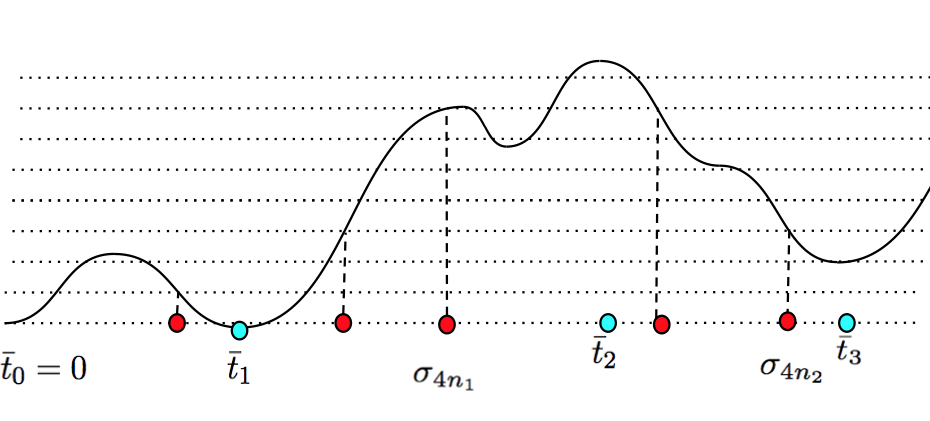}
\caption{Red points represent time points of the form $\sigma_{4k}$.
In this example, $n_1=3$ and $n_2=5$.}
\end{figure}

\begin{lemma}\label{deltat}
$$
\bt_2-\bt_1 \leq \sigma_{4 n_{2}}.
$$
\end{lemma}
\begin{proof}
Let $i_k$ be the index such that 
$ \sigma_{4n_k}\in[\bt_{i_k},\bt_{i_k+1})$. It is sufficient to show that 
(1) $i_{1}\geq 1$, and (2)
$i_{1} \neq i_{2}$.

Let us first deal with (1). By definition of $n_1$, we must have $w(\sigma_{4n_{1}})-\inf_{[0,\sigma_{4n_{1}}]} w \geq h$, which implies that $t_1\leq \sigma_{4 n_{1}}$.
Since $\bt_1 \leq   t_{1}$, we get that $\bt_1\leq \sigma_{4 n_{1}}$ and thus $i_{1}\geq1$.


We now proceed with (2). We claim that if $k$ is such that $\Delta_{4n_k}=-1$, then $\sigma_{4 n_k}$ does not belong to $[\bt_i,\bt_{i+1})$ for every odd integer 
$i$. 
Recall that if $i$ is odd
$$
t_{i+1} = \inf\{u \geq t_{i} \ : \ \sup_{[t_i,u]} w - w(u) = h \}, \  \mbox{and} \ \bt_{i+1}= \sup\{u \geq t_{i} \ : \ \sup_{[t_i,u]} w = w(u)  \}
$$
On the one hand, it is easy to check that $w$ attains its only minimum at $\bt_i$ on the interval $[\bt_{i}, \bt_{i+1}]$. On the other hand,
on the interval $[\sigma_{4(n_k-1)}, \sigma_{4 n_k}]$, $w$ attains its minimum at $\sigma_{4 n_k}$ since $\Delta_{\sigma_{4n_k}} = \cdots = \Delta_{\sigma_{4n_k-3}}=-1$.
Thus, if $\sigma_{4n_k}\in[\bt_{i},\bt_{i+1}]$, we would have $\sigma_{4(n_k-1)}\geq \bt_i$ and 
\beqnn
\sup_{[t_i,\sigma_{4n_k}]} w - w(\sigma_{4 n_k})  &  = &  \sup_{[\bt_i,\sigma_{4n_k}]} w - w(\sigma_{4 n_k})  \\
                                                                                 & \geq  &  \sup_{[\sigma_{4 (n_k-1)},\sigma_{4n_k}]} w   - w(\sigma_{4 n_k})   = h,
\eeqnn
where the first equality follows from the fact that $\sup_{[\bt_i, t_i]} w = w(t_i)$ since $i$ is odd. 
The latter inequality would  imply that $\sigma_{4n_k} \ \geq \  t_{i+1} \ > \bt_{i+1}$, thus yielding a contradiction.

By a symmetric argument, if $k$ is such that $\Delta_{4n_k}=1$, then $\sigma_{4 n_k}\notin[\bt_i,\bt_{i+1}]$ when $i$ even. By definition of $n_{1}$ and
$n_{2}$,
$\Delta_{\sigma_{4n_{1}}}\neq \Delta_{\sigma_{4 n_{2}}}$ and thus,
$\sigma_{4n_{1}}$ and $\sigma_{4n_{2}}$ can not belong to the same interval $[\bt_i,\bt_{i+1}]$. This achieves the proof of claim (2) made earlier, and the proof of our lemma. 
\end{proof}

\begin{corollary}\label{cor:1}
The random variables 
$$
\tau= \bt_4-\bt_2 \ \ \mbox{and} \  \ {\mathcal E} = \int_{\bt_2}^{\bt_3} c(\psi_2'(u)) du \ + \ \int_{\bt_3}^{\bt_4} c(\psi_3'(u)) du
$$
have all finite moments.
\end{corollary}

\begin{proof}
First, $\bt_4-\bt_2$ is equal in distribution
to the sum of two independent copies of $\bt_2-\bt_1$. 
Thus, controlling the moments of  $\bt_4-\bt_2$ 
amounts to controlling the moments of $\bt_2-\bt_1$.
Let $p\in{\mathbb N}$. From the previous lemma, we have
\beqnn
E\left(\bt_2-\bt_1\right)^p & \leq & E( \sigma_{4 n_{2}}^p). 
\eeqnn
By the strong Markov property and by symmetry, the random variable $\sigma_{4 n_{2}}$
is equal in distribution
to the sum of $2$ independent copies of $\sigma_{4n_{1}}$. Thus, it is enough to show that 
$E((\sigma_{4n_{1}})^p)<\infty$. Next,
\beqn\label{666}
E((\sigma_{4n_{1}})^p) & = & \sum_{n\geq0}  \ E[( \sum_{i=1}^{4n} \sigma_i)^p]  \  P(  n_{1} =   n ),
\eeqn
where we used the independence between the $n_k$'s and the $\sigma_i$'s. It is straightforward to check that $\{n_k-n_{k-1}\}_{k\geq 1}$ is a sequence of independent 
geometric random variables with parameter $(1/2)^4$. It then follows that the tail of $n_{1}$
decreases exponentially fast.
Since $\sigma_1$ has all finite moments (by Lemma \ref{moments}), the latter identity implies that 
$\bt_2-\bt_1$ also has all finite moments.

\bigskip

Let us now proceed with the second term.  We will show that $\int_{\bt_1}^{\bt_2} c(\psi_1'(u)) du$
has all finite moments. The moments of  $\int_{\bt_2}^{\bt_3} c(\psi_2'(u)) du$
can be controlled along the same lines.
Recall that we made the assumption that $c$ is locally bounded and 
that there exists $\alpha>0$ such that 
$$
 \lim_{|x|\rightarrow \infty} c(x)/|x|^\alpha = 0.
$$ 
We will assume without loss of generality that $\alpha>1$. 
Let $m,M\geq0$ be such that $|c(x)|\leq m+ M |x|^\alpha$ for every $x\in\R$, so that
$$
|\int_{\bt_1}^{\bt_2} c(\psi_1'(u)) du| \ \leq \ m(\bt_2-\bt_1) +   M \  \int_{\bt_1}^{\bt_2} |\psi_1'(u)|^\alpha du.
$$
Since $\bt_2-\bt_1$ has all finite moments, it remains to control the moments of $\int_{\bt_1}^{\bt_2} |\psi_1'(u)|^\alpha du$.
In order to deal with this term, we will use the so-called free-knot approximation
introduced in \cite{LS15}: Let us consider the function $\phi $ obtained by linear interpolation of the points $(\sigma_n,w(\sigma_n))_{n\geq0}$.
In particuar, 
this function is constructed in such a way that $|| \phi - w ||_{\infty,[0,\infty)}\leq h/2$.
From Corollary \ref{lem:dec}, $\psi_1$ is the unique minimizer of
$$
\Ar\{ \int_{\bt_1}^{\bt_{2}} \  |\vp'(u)|^\alpha \ du \ : \ \vp\in \ \mbox{AC}([\bt_1, \bt_2]),  \ || w-  \vp ||_{\infty, I_1} \leq \frac{h}{2} \}, 
$$
and thus

\beqnn
\int_{\bt_1}^{\bt_2} |\psi_1'(u))|^\alpha du  & \leq & \int_{\bt_1}^{\bt_2} |\phi'(u)|^\alpha du  \\
						      & \leq & \int_{0}^{\sigma_{4n_{2}}} |\phi'(u)|^\alpha du,
\eeqnn
where the last inequality is a direct consequence of Lemma \ref{deltat}.
Further, 
\beqnn
 \int_{0}^{\sigma_{4n_{2}}} |\phi'(u)|^\alpha du  & = &  (\frac{h}{4})^{\alpha} \sum_{i=1}^{4n_{2}}   \frac{1}{ (\sigma_i -\sigma_{i-1})^{\alpha-1}}.
\eeqnn
This yields for any $p\in\mathbb{N}$
\beqnn
E\left( ( \int_{\bt_1}^{\bt_2} |\psi_1'(u)|^\alpha du )^p\right)  & \leq & 
(\frac{h}{4})^{\alpha p} \ E\left(  \left( \sum_{i=1}^{4n_{2}}  \ \frac{1}{ (\sigma_i -\sigma_{i-1})^{\alpha-1}}  \right)^p \right). 
\eeqnn
Using Lemma \ref{moments} and again the fact that $n_{1},n_{2}-n_{1}$ are independent geometric random variables
independent of the $\sigma_i$'s (as in (\ref{666})), 
we get that $\int_{\bt_1}^{\bt_2} \ |\psi_1'(u)|^\alpha du$ has all finite moments. This ends the proof of 
the lemma.

\end{proof}

We now turn to the tightness of $\{R(T)\}_{T\geq0}$ as defined in (\ref{R(T)}).
\subsection{Tightness of $R(T)$}

\begin{lemma}\label{njk}
The sequences of random variables 
$$\left\{\int_{\bt_{N(T)-1}}^T c(\eta_T'(u)) du\right\}_{T\geq0} \ \mbox{and}  \  \ \left\{ \int_{\bt_{N(T)-1}}^{\bt_{N(T)}}  c(\psi_{N(T)-1}'(u)) du \right\}_{T\geq0}$$ 
are tight.
\end{lemma}

\begin{proof}

{\bf Step 0.} We start by recalling a standard result from renewal theory.
Define $\Delta \bt_{i}=\bt_{i+1}-\bt_i$. 
$\Delta \bt_i$ is obviously non-lattice and from the previous subsection, it has a finite first finite moment. Since the $\Delta \bt_i$'s form a 
sequence of i.i.d. random variables, from \cite{M74}, the random sequence
$$
(T-\bt_{N(T)}, \Delta \bt_{N(T)}, \cdots, \Delta \bt_{(N(T)-k)^+},\cdots )
$$
converges (in the sense of finite dimensional distributions) to the sequence
$$
(u, \delta \bt_0, \cdots, \delta \bt_k, \cdots)
$$
where the $\delta \bt_i$'s are independent; for $i\geq 1$ the r.v. $\delta \bt_i$
is distributed as $\Delta \bt_2$; $\delta \bt_0$ has the size biased distribution:
$$
E(f(\delta \bt_0)) \ = \ \frac{1}{E(\Delta \bt_2)}\int_{\R^+} f(u) P(\Delta \bt_2>u) du
$$
for every test function $f$ (i.e., infinitely differentiable with compact support). Finally, $u$
is independent of the $\delta \bt_i$'s with $i\geq1$, but conditioned on $\delta \bt_0$,
$u$
is a uniform random variable on $[0,\delta \bt_0]$ (in the renewal terminology, $u$ has the backward recurrence 
time distribution).

\bigskip

{\bf Step 1.} There exists $m,M\geq0$ and $\alpha>1$ such that 
$
c(u) \leq m + M |u|^\alpha.
$
Thus
$$
|\int_{\bt_{N(T)-1}}^T c(\eta_T'(u)) du| \ \leq  \ m(T-\bt_{N(T)-1}) + M \int_{\bt_{N(T)-1}}^T |\eta_T'(u)|^\alpha du.
$$
Further,
\begin{eqnarray*}
|\int_{\bt_{N(T)-1}}^{N(T)} c(\psi_{N(T)-1}'(u)) du 
& \leq & m(\bt_{N(T)} - \bt_{N(T)-1})  \ + \  M \int_{\bt_{N(T)-1}}^{\bt_{N(T)}}  |\psi_{N(T)-1}'(u)|^\alpha du  \\
& \leq & m(\bt_{N(T)} - \bt_{N(T)-1}) \ + \ M  \int_{\bt_{N(T)-1}}^{\bt_{N(T)}}  |\eta_{T}'(u)|^\alpha du  \\
& \leq & m(T - \bt_{N(T)-1})  \ + \  M \int_{\bt_{N(T)-1}}^T  |\eta_{T}'(u)|^\alpha du.
\end{eqnarray*}
where the second inequality is a direct consequence of Corollary \ref{lem:dec}. Thus, the two sequences of interest
are bounded from above by the RHS of the latter inequality. Finally,
since Step 0 above implies that 
the first term converges in distribution to $m (u+\delta t_1)$, it remains to show the tightness of 
 $\{\int_{\bt_{N(T)-1}}^T  |\eta_{T}'(u)|^\alpha du\}_{T\geq0}$.

\bigskip

{\bf Step 2.}
Define $\theta(T) = \sup\{n \ : \ \sigma_n \leq T\}$ 
and consider the function $f_T$ obtained by linear interpolation of the points
$$
\{(\sigma_k , w(\sigma_k))\}_{k\leq \theta(T)}  \ \mbox{and} \ (T,w(T)),
$$
in such a way that $f_T$ is an admissible function 
for the minimization problem (\ref{Min-problem}). 
From Theorem \ref{main}, for $N(T)\geq 4$, we have
\beqnn
\eta_T(\bt_{N(T)-2}) - f_T(\bt_{N(T)-2}) \ = \ w(\bt_{N(T)-2}) - f_T(\bt_{N(T)-2}) - (-1)^{N(T)-2} \frac{h}{2} \\
\eta_T(\bt_{N(T)-1}) - f_T(\bt_{N(T)-1}) \ = \ w(\bt_{N(T)-1}) - f_T(\bt_{N(T)-1}) - (-1)^{N(T)-1} \frac{h}{2} 
\eeqnn
Since $f_T$ and $w$ must stay with a distance $h/2$ from one another, the intermediate value theorem
implies that  
there exists $u_{N(T)-2}\in[\bt_{N(T)-2}, \bt_{N(T)-1}]$
such that $f_T(u_{N(T)-2})=\eta_T(u_{N(T)-2})$. Thus,
\beqnn
\int_{\bt_{N(T)-1}}^T  |\eta_T'(u)|^\alpha \ du  & \leq &  \int_{u_{N(T)-2}}^T |\eta_{T}'(u)|^\alpha du \\
								    & \leq & \int_{u_{N(T)-2}}^T |f_T'(u)|^\alpha du 
								     \leq  \int_{\bt_{N(T)-2}}^T |f_T'(u)|^\alpha du 
\eeqnn
where we also used the fact that $\eta_T$ minimizes the energy on $[0,T]$,
and thus that $\eta_T$ is also the only minimizer of
\beqnn
\Min\{\int_{u_{N(T)-2}}^T |\vp'(u)|^\alpha du \ : \ \vp\in \mbox{AC}([u_{N(T)-2}, T]), \\
\vp(u_{N(T)-2}) = f_T(u_{N(T)-2}), \ \vp(T)=f_T(T) ,||\vp-w||_{\infty,[u_{N(T)-2}, T]}\leq \frac{h}{2} \}.
\eeqnn
Let us now assume that there exists $K$ large enough such that 
$0\leq \sigma_{\theta(T)-K} \leq\bt_{N(T)-2}$. The previous inequality implies 
that
\beqnn
\int_{\bt_{N(T)-1}}^T  \ |\eta_T'(u)|^\alpha \ du
 &  \leq  & 
 \int_{\sigma_{\theta(T)-K}}^T |f_T'(u)|^\alpha du  	 \\
  &  =    &
\sum_{k=0}^{K-1} \int_{\sigma_{\theta(T)-(k+1)}}^{\sigma_{\theta(T)-k}} |f_T'(u)|^\alpha du  \ + \ \int_{\sigma_{\theta(T)}}^T |f_T'(u)|^\alpha du  \\
								     & \leq & 
								     \frac{h^\alpha}{4^\alpha} 
								     \left( 
								     \sum_{k=0}^{K-1} \frac{1}{(\Delta \sigma_{\theta(T)-(k+1)})^{\alpha-1}} \ + \ 
								      \frac{1}{(T-\sigma_{\theta(T)})^{\alpha-1}} \right)		
\eeqnn				
where $\Delta \sigma_i = \sigma_{i+1}-\sigma_i$.

\bigskip

{\bf Step 3.} By applying the same renewal theorem used in Step 0 (applied now to the sequence $\{\sigma_n\}_{n\geq0}$), the RHS of the latter inequality is tight. Thus, we need to show that 
$$
\lim_{K\uparrow\infty} \ \lim_{T\uparrow\infty} P\left(  0 \leq  \sigma_{\theta(T)-K}\leq \bt_{N(T)-2}   \right) \ = \ 1.
$$
Recall that $\Delta_{n} \ = \ \frac{4}{h}(w(\sigma_{n+1})-w(\sigma_n))$.
Let us assume that we can find $1<i_1<i_2<i_3<i_4<K-3$ such that
$$
\Delta_{\theta(T)- i_k}=\cdots=\Delta_{\theta(T)-(i_k-3)} = (-1)^k
$$
so that $w$ experiences a downcrossing (resp., upcrossing) of size $h$ for $k$ odd (resp. even)
on the interval $[\theta(T)-i_k, \theta(T)-(i_k-4)]$. 
By reasonning as in Lemma \ref{deltat}, the times $\sigma_{\theta(T)-i_k}$'s
 must all belong to distinct intervals $[\bt_{i}, \bt_{i+1}]$. Since $\theta(T) \leq T$, this yields
$$
\sigma_{\theta(T)-i _4} \leq \bt_{N(T)-2}.
$$
On the other hand, by independence of the $\sigma_k$'s and the $\Delta_k$'s, the sequence
$$
\left( \ \Delta_{(\theta(T)-1)^+},\cdots,\Delta_{(\theta(T)-K)^+}, \cdots \ \right)
$$
converges in distribution to the first $K$ coordinates of an infinite sequence of i.i.d. random variables 
$
(\delta_{1},\cdots, \delta_{K},\cdots)
$ 
with $P(\delta_i=\pm 1)=1/2$. Since the probability to find $4$ indices $1<i_1<\cdots<i_4$
such that 
$$
\delta_{i_k}=\cdots=\delta_{i_k-3} = (-1)^k
$$
in the infinite sequence
$
\{\delta_{k}\}
$ 
is equal to $1$, it follows that 
$$
\lim_{K\rightarrow\infty}\lim_{T\rightarrow\infty} \ P\left(  \sigma_{\theta(T)-K}\leq \bt_{N(T)-2}   \right) \ = \ 1.
$$

\end{proof}

\begin{corollary}\label{cor:tightness}
The sequence $\{R(T)\}_{T\geq0}$ is tight.
\end{corollary}
\begin{proof}
From the devious result,  it remains to control the first term of $R(T)$, i.e., $\int_{0}^{\bt_2} c(\eta_T'(u)) du$. When $N(T)\geq4$, 
Theorem \ref{main} implies that $\eta_T(\bt_2)=w(\bt_2)-h/2$ and thus 
the taut string $\eta_T$
restricted on the interval $[0,\bt_2]$ must be the unique solution of 
\begin{eqnarray*}
\Min\{ \int_{0}^{\bt_2} \ |\vp'(u)|^2 \ du \ : \ \vp\in \ \mbox{AC}([0, \bt_2]),\  
\vp(0)=0 \ ,  \ \  \vp(\bt_2) = w(\bt_2) - \frac{h}{2}, \    
|| w-  \vp ||_{\infty,[0,\bt_2]} \leq \frac{h}{2}  \}.
\end{eqnarray*}
This implies that  $\int_{0}^{\bt_2} c(\eta_T'(u)) du$  is constant, if $T$ is large enough so that  $N(T)\geq4$. This obviously entails  tightness of the first term.
This ends the proof of our corollary.
\end{proof}

\bigskip

\section{Proof of Proposition \ref{CTLA}}
\label{proof:CLTA}

Let $[x]$ denote the integer part of $x$. Write
\begin{eqnarray*}
\tS_{\tilde{\mathcal N}(t)-1} \ - \  t \frac{\bX}{\bT} \ = \ (\tS_{\tilde{\mathcal N}(t)-1} - \tS_{[\frac{t}{\bT}]}) \ + \ (\tS_{[\frac{t}{\bT}]} - t\frac{\bX}{\bT}).
\end{eqnarray*}
Our proposition is a direct consequence of the two following lemmas.

\begin{lemma}\label{kla}
$$
\lim_{t\rightarrow\infty} \ \frac{1}{\sqrt{t}} \left( \ (\tS_{\tilde{\mathcal N}(t)-1} \ - \ \tS_{[\frac{t}{\bT}]}) \ - \  \bX(\tilde{\mathcal N}(t)-[\frac{t}{\bT}]) \ \right)\ = \  0 \ \mbox{in probability,}
$$
\end{lemma}

\begin{lemma}\label{renew1}
The pair of random variables
$$
\left( \frac{1}{\sqrt{t}} \ (\tilde{\mathcal N}(t)-t/\bT) \cdot \frac{\bT^{3/2}}{\sigma_\tau } \ , \ (\tS_{[\frac{t}{\bT}]} - t \frac{\bX}{\bT}) \cdot \frac{\sqrt{\bT/t}}{\sigma_X} \right)
$$
converges in distribution to a two dimensional Gaussian random vector with mean $0$ and 
covariance matrix $\Sigma=\left( \begin{array}{cc} 1 & -\rho \\ -\rho & 1 \end{array} \right)$,
where $\rho=\sigma_{X,\tau}/\sigma_X\sigma_\tau$.
\end{lemma}

\bigskip

\begin{proof}[Proof of Lemma \ref{kla}]
Define $Z_n = U_n - \bar X n$.
Let $\eps\in(0,1)$ and define $n_0(t) = [t/\bar \tau]$, $n_1(t) = [t/\bar \tau\cdot(1-\eps^3)]$, $n_2(t) = [t/\bar \tau\cdot(1+\eps^3)]$.
We aim at showing that
$\frac{1}{\sqrt{t}} \left( \ Z_{n_0(t)}  - Z_{\tilde{\mathcal N}(t)-1}\  \right)$ converges to $0$ in probability.
\begin{eqnarray*}
P\left( \frac{1}{\sqrt{t}} | Z_{n_0(t)}  - Z_{\tilde{\mathcal N}(t)-1}\ | >\eps  \right) 
& \leq & 
P\left( \frac{1}{\sqrt{t}} \ |Z_{n_0(t)}  - Z_{\tilde{\mathcal N}(t)-1}| \  >\eps, \ \tilde{\mathcal N}(t) \in[n_1(t),n_2(t)] \right) \  \\ 
&& + \ P\left(\tilde{\mathcal N}(t)\notin [n_1(t), n_2(t)] \right) 
\end{eqnarray*}
The second term on the RHS goes to $0$ as $n$ goes to $\infty$ by a standard renewal theorem. 
As for the first term, we use Kolmogorov inequality,
\begin{eqnarray*}
P\left( \frac{1}{\sqrt{t}} | \ Z_{n_0(t)}  - Z_{\tilde{\mathcal N}(t)-1}\  | \geq \eps, \ \tilde{\mathcal N}(t) \in[n_1(t),n_2(t)] \right) 
& \leq &
P\left(\sup_{u\in\{n_1(t)-1, \cdots, n_0(t)-1\}}  \frac{1}{\sqrt{t}} | \ Z_u  - Z_{n_0(t)}\  | \geq \eps \right) \\
&&+
P\left(\sup_{u\in\{n_0(t)-1,\cdots,n_2(t)-1\}}  \frac{1}{\sqrt{t}} | \ Z_u  - Z_{n_0(t)}\  | \geq \eps \right) \\
& \leq & \frac{1}{\eps^2 t} \ \left( E( Z_{n_0(t) - n_1(t)}^2) \ + \ E( Z_{n_2(t) - n_0(t)}^2)  \right)  \\
& \leq &  \frac{ 2 \sigma_X^2}{\eps^2 t} \frac{t}{\bar \tau} \eps^3 = 2\frac{\sigma_X^2}{\bar \tau} \eps. 
\end{eqnarray*}
Since the RHS goes to $0$ with $\eps$, this completes the proof of the lemma.
\end{proof}

\bigskip

\begin{proof}[Proof of Lemma \ref{renew1}]
To ease the notations, we write
$$
Y_t \ = \  (\tS_{[\frac{t}{\bT}]} - t \frac{\bX}{\bT}) \cdot \frac{\sqrt{\bT/t}}{\sigma_X}
$$
For any $y\in\R$ and $t\in\R^+$, define $n_t^y$ to be the integer part of the unique positive solution (in $x$)
of the equation
\begin{eqnarray}\label{def:t}
t &  =  & x \bar \tau + \sqrt{x} y \sigma_\tau . 
\end{eqnarray}
Let $a<b, c<d$ be four arbitrary real numbers.
First, for every $y$
\begin{eqnarray}
n_t^y & = & \frac{t}{\bT} - \sqrt{t} y \sigma_\tau / \bT^{3/2} + o(\sqrt{t}). \nonumber 
\end{eqnarray}
and thus there exists $\eps_t^a,\eps_t^b$ such that $\eps_t^a/\sqrt{t},\eps_t^b/\sqrt{t}\rightarrow\infty$ as $t\rightarrow0$ and 
\begin{eqnarray}
P\left(\tilde{\mathcal N}(t)\in[n_t^b,n_t^a), \ \  Y_t \in [c,d] \right) 
& = &
P\left( \tilde{\mathcal N}(t) \in [t/\bT - \frac{\sqrt{t} b \sigma_\tau}{\bT^{3/2}} + \eps_t^b, \  t/\bT - \frac{\sqrt{t} a \sigma_\tau}{\bT^{3/2}}+\eps^a_t], \ Y_t\in[c,d] \right) \nonumber \\ 
& = &
P\left( -\frac{1}{\sqrt{t}} \ (\tilde{\mathcal N}(t)-t/\bT) \cdot \frac{\bT^{3/2}}{\sigma_\tau } \in[a-\frac{\tilde \eps_t^a}{\sqrt{t}},b - \frac{\tilde \eps_t^b}{\sqrt{t}}], \ Y_t\in[c,d] \right) \label{erto}
\end{eqnarray}
where  $\tilde \eps_t^a/\sqrt{t},\tilde \eps_t^b/\sqrt{t}\rightarrow 0$ as $t\rightarrow\infty$.

\medskip

Secondly, let us now evaluate the law of the LHS of (\ref{erto}).
\begin{eqnarray*}
P\left(\tilde{\mathcal N}(t)\in[n_t^b,n_t^a), \ \  Y_t \in [c,d] \right) 
& = &
P\left( S_{n_t^b} \leq t , \ S_{n_t^a} > t , Y_t \in [c,d] \right) \nonumber  \\
& = &
P\left( \frac{S_{n_t^b} - n^b_t\bT}{\sigma_\tau \sqrt{n^b_t}} \leq \frac{t - n^b_t\bT}{\sigma_\tau \sqrt{n^b_t}}, \ \frac{S_{n_t^a} - n_t^a\bT}{\sigma_\tau \sqrt{n^a_t}} > \frac{t - n_t^a\bT}{\sigma_\tau \sqrt{n^a_t}} , Y_t \in [c,d] \right) \nonumber  \\
& = &
P\left( \frac{S_{n_t^b} - n^b_t\bT}{\sigma_\tau \sqrt{n^b_t}} \leq b, \ \frac{S_{n_t^a} - n_t^a\bT}{\sigma_\tau \sqrt{n^a_t}} > a, Y_t \in [c,d] \right) \nonumber 
\end{eqnarray*}
where the third inequality follows directly from the definition of $n_t^y$.
Since $n_t^a, n_t^b \approx t/\bT$, it is tempting to use this approximation in the latter equality
and write 
\begin{eqnarray*}
P\left(\tilde{\mathcal N}(t)\in[n_t^b,n_t^a), \ Y_t \in [c,d]  \right) 
& \approx &
P\left( \frac{S_{[t/\bT]} - t}{\sigma_\tau \sqrt{t/\bT}} \in (a,b], Y_t \in [c,d] \right), \nonumber
\end{eqnarray*}
More formally,
\begin{eqnarray*}
P\left(\tilde{\mathcal N}(t)\in[n_t^b,n_t^a), \ \ Y_t \in [c,d] \right) 
&  = &
P\left( \frac{S_{[t/\bT]} - t}{\sigma_\tau \sqrt{t/\bT}} \leq b - \bar \eps_t^ b, \ \frac{S_{[t/\bT]} - t}{\sigma_\tau \sqrt{t/\bT}} > a - \bar \eps_t^a, Y_t \in [c,d] \right) \nonumber \nonumber
\end{eqnarray*}
where 
$$
\forall y=a,b, \ \  \bar \eps_t^y \ = \ 
\frac{1}{\sqrt{t}}\left( S_{n_t^y} - S_{[t/\bT]} \ - \ \bT(n_t^y - t/\bT) \right) 
+  \ \frac{1}{\sqrt{t}}\left(  S_{n_t^y}  - \bT n_t^y \right) (\sqrt{\frac{t/\bT}{n^t_y}}-1). 
$$
We now claim that  $\bar \eps_t^a$ and $\bar \eps_t^b$ vanish as $t$ goes to $\infty$.
Indeed, applying Markov inequality twice yields that for every $\delta>0$
\begin{eqnarray*}
P( \bar \eps_t^y > \delta) & \leq &
\frac{1}{\delta^2 t } |n_t^y - [t/\bT]| E(\tau^2) 
+ \frac{n_t^y}{\delta^2 t}  (\sqrt{\frac{t/\bT}{n_t^y}}-1 )^2  E(\tau^2)  
\end{eqnarray*}
and thus $
\lim_{t\uparrow\infty}  \ \bar \eps_t^y   =   0
$
in probability. On the other hand, by
the multidimensional CLT,
$
\left( \frac{S_{[t/\bT]} - t}{\sigma_\tau \sqrt{t/\bT}}, Y_t \right)
$
converges in distribution to the two dimensional gaussian vector with mean $0$
and correlation matrix  $\left( \begin{array}{cc} 1 & \rho \\ \rho & 1 \end{array} \right)$. 
It then follows that for every $a,b,c,d$
$$
P\left(\tilde{\mathcal N}(t)\in[n_t^b,n_t^a), \ \   (\tS_{[\frac{t}{\bT}]} - t \frac{\bX}{\bT}) \cdot \frac{\sqrt{\bT/t}}{\sigma_X} \in [c,d] \right) \rightarrow \int_{[a,b]\times[c,d]} \exp\left( - \frac{1}{2} \ {^tx} \Sigma^{-1} x \right) \frac{1}{2 \pi \mbox{det}(\Sigma)^{1/2}} d\lambda(x)
$$
where $\lambda$ is the Lebesgue measure on $\R^2$. 
Combining this with (\ref{erto}) then yields our result.

\end{proof}

\section{Appendix}

\begin{prop}\label{ghk}
Let $0<a<b$ and $c,d\in\R$ such that $|c-w(a)|, |d-w(b)| \ \leq\  h/2$. If $C$ is a strictly convex function then both sets
\beqnn
\Ar\{ \int_a^b \ C(\vp'(t)) \ dt \ : \ \vp\in \ \mbox{AC}([a,b]), \  || w-  \vp ||_{\infty,[a,b]} \leq \frac{h}{2} \} \\
\Ar\{ \int_a^b \ C(\vp'(t)) \ dt \ : \ \vp\in \ \mbox{AC}([a,b]),  \ \vp(a) = c, \ \vp(b)=d, \  || w-  \vp ||_{\infty,[a,b]} \leq \frac{h}{2} \} \\
\eeqnn
have a unique element.
\end{prop}
\begin{proof}
This is a rather standard result in convex analysis. For a proof, we refer
to Lemma 2 in \cite{G07}. In this reference, the result is shown in the particular case where $C(x)=\sqrt{1+x^2}$ with
fixed boundary conditions.
We let the reader convince herself that the same proof applies for any convex function and also for the analogous minimization problem with free boundary conditions.
\end{proof}

\begin{prop}\label{prop-der}
Let $0<a<b$ and $c,d\in\R$
such that $|w(a)-c|\leq h/2$ and $|w(b)-d|\leq h/2$. Finally, let $V$ be the unique solution of the minimization problem,
$$
\Min\{ \int_a^b \ |\vp'(u)|^2 \ du \ : \ \vp\in \ \mbox{AC}([a,b]),\ \vp(a)=c \ , \vp(b)  = d ,   || w-  \vp ||_{\infty,[a,b]} \leq \frac{h}{2}  \}.
$$

For any strictly convex function $C$, 
the function $V$ is also the unique minimizer of the following minimization problem:
\beqn\label{Min-problem2}
\Min\{ \int_a^b \ C(\vp'(u)) \ du \ : \ \vp\in \ \mbox{AC}([a,b]),\ \vp(a)=c \ , \vp(b)  = d ,   || w-  \vp ||_{\infty,[a,b]} \leq \frac{h}{2}  \}, 
\eeqn
\end{prop}

Before proceeding with the proof of the proposition, we note that 
an analogous result in the discrete setting can be found  in \cite{SGGHL09} -- see Theorem 4.35 and Theorem 4.46
therein. As we shall now see, the latter proposition is also implicit in Grasmair \cite{G07}.
In the following, we fix a strictly convex function $C$ and we will denote by $U$ the unique solution of (\ref{Min-problem2}). We will now  show that $U$
and $V$ must coincide.

\begin{lemma}
If $u\in L^1([a,b])$ denotes the derivative of $U$, 
then $u$ is of local bounded variation on $(a,b)$, i.e., for every $(s,t)\subset [a,b]$
the total variation of $u$ on $(s,t)$ is finite. 
\end{lemma}

\begin{proof}
We follow closely the proof of Proposition 2 in \cite{G07}. We provide an argument by contradiction.
Let us assume that there exists $t\in(a,b)$ such that for every $\delta>0$, the total variation 
of $u$ on $(t-\delta,t+\delta)$ is infinite.
We distinguish between two cases: either Let us first assume that 
 $U(t) < w(t)+ \frac h 2$ or $U(t) > w(t)- \frac h 2$. We will only deal with the first case, since the second case is completely analogous.

We can find 
$s\in(t-\delta,t+\delta)$ such that 
$$
\mbox{ess inf}\{u(x), \ x\in(t-\delta,s)\} \leq \mbox{ess sup}\{u(x), \ x\in(s, t+\delta)\}
$$
else $u$ would be monotone and thus of finite variation. Thus, we can find $T_1 \subset (t-\delta,s)$
and  $T_2 \subset (s, t+\delta)$ such that $|T_1|=|T_2|>0$ (where $|T_i|$ denotes the Lebesgue measure 
of the set $T_i$) and 
$$
\mbox{ess inf}\{u(x), \ x\in T_1 \} \leq \mbox{ess sup}\{u(x), \ x\in T_2\}.
$$
Define 
$$
h(x)=
\left\{
\begin{array}{cc} 
1 & \mbox{if $x\in T_1$}\\
-1 & \mbox{if $x\in T_2$}\\
0 & \mbox{otherwise}
\end{array} \right.
$$
and let $H(t)=\int_{a}^t h(s) ds$. Since $U(s)<w(s)+h/2$ on $(t-\delta, t+\gamma)$, for $\gamma>0$ small enough,
the function $\gamma H$ belongs to the set
$$
\{Y\in \mbox{AC}(a,b) \ :  \ Y(a)=0, \ Y(b)=0, \  || U + Y - w ||_{\infty,[a,b]} \leq \frac h 2\},
$$
i.e., $U+\gamma H$ is an admissible function for the variational problem at hands.
Further, the G\^ateau derivative of $\vp\rightarrow \int_{a}^b C(\vp'(x)) dx$ evaluated at $U$ in the direction $H$
is given by
$$
\int_{T_1} C'(u(x)) dx - \int_{T_2} C'(u(x)) dx.
$$ 	
(note the latter expression is well defined since $C$ is a real convex function and thus is absolutely continuous). Since $C'$ is strictly increasing, the choice of $T_1$ and $T_2$ induces that this derivative is strictly negative,
which contradicts the minimality of $U$. This shows that for every $t$ with $U(t) < w(t)+ \frac h 2$, 
there exists $\delta>0$ such that the total variation of $u$ on $(t-\delta, t+\delta)$ is finite. By a symmetric argument,
one can show that the same property holds when $t$ is such that $U(t) > w(t)+ \frac h 2$. This ends the proof of our lemma.

\end{proof}

Since $u$ is of local bounded variation, there exists a Radon measure $Du$ satisfying
the relation
$$
\int_a^b \ \psi'(x) u(x) dx \ = \ -  \int_a^b \ \psi(x) Du(dx) 
$$
for every function $\psi\in C^\infty_c(a,b)$ -- the set of infinitely differentiable functions with compact support on $(a,b)$. Let $|Du|$ 
denote the total variation of $Du$ (see again \cite{G07} for more details). Since $u$ is of local bounded variation (by the previous lemma), 
the Radon-Nikodym derivative $d Du/d |Du|\in L^1(a,b)$ is defined $Du$-almost surely on $(a,b)$, and takes the value $\pm1$.
(again $Du$ almost surely). (For more details, we again we refer to \cite{G07}.)

\begin{lemma}\label{lem12}
$U$ satisfies the following constraints
\begin{enumerate}
\item $U\in \mbox{AC}(a,b)$ .
\item $U(a)=c, \ U(b)=d$.
\item $||U - w ||_{\infty,[a,b]} \leq h/2$.
\item $u$ is of bounded local variation on $(a,b)$ and further $U(t)= w(t) + \frac h 2 \frac{d D u }{d |Du|}(t)$ $Du$-a.s on $(a,b)$.
\end{enumerate}
\end{lemma}

When $u$ is a smooth function, $d Du/d |Du|$ coincides with $\mbox{sign}(u')$. Thus, the last condition can be interpreted as follows: away from the boundary of the tube, the function $U$ is taut (i.e. $u'=0$), whereas the only possibility for the function $U$ to bend upwards (resp., downwards) is when $U$
touches the upper part of the tube (resp., lower part of the tube), i.e., $U(t)= w(t)+\frac h 2$ (resp., $U(t)= w(t)-\frac h 2$).

\begin{proof}[Proof of Lemma \ref{lem12}]

The first three properties directly follow from the definition of $U$.
The previous lemma implies that $u$ is of local bounded variation and it only remains to show that 
$U(t)=w(t) + \frac h 2 \frac{d D u }{d |Du|}(t)$ $Du$-a.s on $(a,b)$.

Again we follow closely Proposition 2 in \cite{G07}. Let $t\in(a,b)$ be a Lebesgue point of the function 
$d Du/d |Du|$ (with respect to the measure $|Du|$) such that $dDu/d|Du |=1$. We need to show that 
$U(t)=w(t)+\frac h 2$. Let us assume that $U(t) < w(t)+\frac h 2$. Since $t$ is a Lebesgue point with respect to $|Du|$,
we have
$$
\lim_{\delta\rightarrow 0} \frac{\int_{t-\delta}^{t+\delta} \  | \ d Du/d|Du|(s)-1 \ | \ | Du |(ds)}{|Du|(t-\delta,t+\delta)} = 0 
$$
As a consequence, for $\delta$ small enough,we have
$$
u(t+\delta) = u(t-\delta) +  \int_{t-\delta}^{t+\delta} \frac{d Du}{d |Du|}(s) d |Du|(s) \ > \ u(t-\delta).
$$
As in the previous lemma, this entails
the existence of $T_1\subset (t-\delta,t), T_2\subset (t,t+\delta)$ such that 
such that $|T_1|=|T_2|>0$ and 
$$
\mbox{ess inf}\{u(x), \ x\in T_1 \} \leq \mbox{ess sup}\{u(x), \ x\in T_2\}.
$$
By the same reasoning as in the previous lemma, this contradicts the minimality of $U$, and
thus $U(t)=w(t)+\frac h 2$. 

By the same argument, one can show if  $t\in(a,b)$ a Lebesgue point of the function 
$d Du/ d |Du|$ with respect to the measure $|Du|$ such that $d Du/d |Du |=-1$, then
$U(t)=w(t)-\frac h 2$. This completes the proof of our lemma.

\end{proof}

\begin{proof}[Proof of Proposition \ref{prop-der}]
Using Proposition 3 in \cite{G07}, there is a unique function satisfying the constraints (1)--(4) of 
Lemma \ref{lem12}. On the other hand, those constraints do not depend on the strictly convex function $C$ at hands.
Since $x \rightarrow x^2$ is also strictly convex, it follows that $U=V$.
\end{proof}

\bigskip

{\bf Acknowledgments.} I am indebted to M. Lifshits for introducing me to this topic and for valuable discussions.

\end{document}